\documentclass[12pt,reqno]{amsart}

\usepackage[letterpaper,textwidth=6.9in,textheight=9.6in,centering]{geometry}

\usepackage[T1]{fontenc}
\usepackage{lmodern}
\usepackage{microtype}
\usepackage{amsmath,amssymb,mathtools,amsthm}
\numberwithin{equation}{section}
\usepackage{mathrsfs}
\usepackage{enumitem}
\usepackage[colorlinks=true,linkcolor=blue,citecolor=blue,urlcolor=blue]{hyperref}

\newtheorem{theorem}{Theorem}[section]
\newtheorem{proposition}[theorem]{Proposition}
\newtheorem{lemma}[theorem]{Lemma}

\theoremstyle{plain}
\newtheorem{kwconjecture}{Conjecture}

\theoremstyle{remark}
\newtheorem{remark}[theorem]{Remark}

\DeclareMathOperator{\conv}{conv}
\DeclareMathOperator{\supp}{supp}
\DeclareMathOperator{\Vol}{Vol}
\newcommand{\R}{\mathbb R}
\newcommand{\1}{\mathbf 1}
\newcommand{\ip}[2]{\langle #1,#2\rangle}

\newcommand{\dd}{\,\mathrm d}
\renewcommand{\S}{\mathcal{S}}
\newcommand{\D}{\mathrm{D}}
\newcommand{\cl}{\operatorname{cl}}

\title[The many-body Blaschke-Santal\'o type inequality via optimal transport]{The many-body Blaschke-Santal\'o type inequality via optimal transport}
\author{Shibing Chen}
\address{School of Mathematical Sciences,
University of Science and Technology of China,
Hefei, Anhui 230026, China}
\email{chenshib@ustc.edu.cn}

\author{Yuanyuan Li}
\address{Institute for Theoretical Sciences,
Westlake University, Hangzhou, 310030, China}
\email{lyyuan@westlake.edu.cn}

\author{Dongmeng Xi}
\address{Department of Mathematics,
Shanghai University,
Shanghai 200444, China}
\email{xi\_dongmeng@shu.edu.cn; dongmeng.xi.math@gmail.com}

\author{Zhe-Feng Xu}
\address{School of Mathematical Sciences,
University of Science and Technology of China,
Hefei, Anhui 230026, China}
\email{xzf1998@mail.ustc.edu.cn}
\date{\today}

\begin{document}

\begin{abstract}
Let $K_1,\ldots,K_k\subset\mathbb R^n$ be origin-symmetric measurable sets of finite volume such that
\[
 \sum_{1\le i<j\le k}\langle x_i,x_j\rangle\le \binom{k}{2},
 \qquad  \forall\,x_i\in K_i, x_j\in K_j.
\]
We prove the sharp many-body Blaschke--Santal\'o type inequality
\[
 \prod_{i=1}^k |K_i|\le |B^n|^k
\]
proposed by Kalantzopoulos and Saroglou,  and characterize all equality cases.

The proof combines multi-marginal optimal transport with a  pseudo-Euclidean volume estimate.
Using the geometric--functional equivalence of Kalantzopoulos and Saroglou, we also establish the functional version  inequality  proposed by Kolesnikov and Werner.
\end{abstract}

\maketitle

\section{Introduction and main results}

Let \(K\subset\R^n\) be a convex body containing the origin in its interior.  Its polar body is
\[
 K^\circ=\{y\in\R^n:\ip{x}{y}\le1 \text{ for all }x\in K\}.
\]
The Blaschke--Santal\'o inequality is a well-known affine isoperimetric inequality.  In the origin-symmetric case, it asserts that
\begin{equation*}
 |K|\,|K^\circ|\le |B^n|^2,
\end{equation*}
with equality if and only if $K$ is an ellipsoid \cite{B1917,Santalo}.

The complementary lower-bound problem for the volume product is Mahler's conjecture, originating in \cite{Mahler,Mahler1939}. Mahler proved the
two-dimensional case, while Bourgain--Milman established the reverse Santal\'o inequality up to a universal constant in arbitrary dimension \cite{BourgainMilman1987}. In dimension three, the symmetric conjecture was proved by Iriyeh--Shibata \cite{IS2020}, and the general conjecture was
recently proved in \cite{CLXXMahler}.

The purpose of this paper is to prove a sharp upper bound for the many-body volume product and to show that its equality structure becomes rigid once \(k\ge 3\). 
The many-body problem arose from the functional extension of the Blaschke--Santal\'o inequality introduced by Kolesnikov and Werner \cite{KW}.  Its optimal-transport origin is more specific. Early many-marginal forms of cyclical monotonicity and optimal couplings appear in the work of Knott--Smith and R\"uschendorf \cite{KnottSmith1994,Ruschendorf1996}. For the Euclidean quadratic cost, the foundational result of Gangbo--Święch shows that the multi-marginal Kantorovich optimizer is unique and is induced by optimal maps \cite{GS}. Agueh and Carlier later observed that the same quadratic multi-marginal problem is equivalent to the Wasserstein barycenter problem \cite{AC}.

The connection with optimal transport comes from the
identity
\[
 \sum_{1\le i<j\le k}|x_i-x_j|^2
 =(k-1)\sum_{i=1}^k|x_i|^2
 -2\sum_{1\le i<j\le k}\ip{x_i}{x_j},
\]
which shows that the pairwise inner products are precisely the interaction variables in the quadratic multi-marginal cost. The pseudo-Riemannian aspect of optimal transport is already present in
the two-marginal theory. Kim and McCann introduced the pseudo-Riemannian geometry induced by the mixed Hessian of the transport cost \cite{KM10}. Building on this framework, Kim--McCann--Warren associated to a smooth
transportation cost a pseudo-Riemannian metric and a calibration form on
the product space; in their formulation, the graph of an optimal map is a
calibrated maximal submanifold \cite{KimMcCannWarren2010}.  In the
multi-marginal problem, Pass introduced semi-Riemannian metrics built from
the mixed second derivatives of the cost and used their signature to
control the local dimension of optimal supports \cite{Pass2012}; see also
\cite{Pass2015}.

In the present paper, the relevant object is the many-body interaction
\[
\S_2(x_1,\ldots,x_k)=\sum_{1\le i<j\le k}\langle x_i,x_j\rangle .
\]
For this interaction, the mixed Hessians are constant:
\[
\D^2_{x_i x_j} \S_2(x_1,\ldots,x_k)=I,
\qquad \mathrm{for}\,\,i\ne j.
\]
Thus the semi-Riemannian geometry of the optimal support becomes flat.
The corresponding bilinear form is
\begin{equation}\label{bilinearform}
 \mathcal B(X,Y)
 =
 \left\langle\sum_{i=1}^k x_i,\sum_{i=1}^k y_i\right\rangle
 -
 \sum_{i=1}^k\ip{x_i}{y_i},
 \qquad X=(x_i),\ Y=(y_i)\in (\mathbb{R}^n)^k,
\end{equation}
and
\[
 \mathcal B(X,X)=2\S_2(X).
\]
This identity is the bridge between the many-body polarity condition and
the spacelike graph theory used below.

Pairwise interaction costs have rich structures and have been studied by many authors. Pass proved uniqueness and Monge solution results under suitable structural hypotheses \cite{Pass2011}. Kim and Pass introduced the twist-on-splitting-sets condition and extended the barycenter cost to Riemannian manifolds \cite{KimPass2014,KimPass2015}. The complete-graph case and more general graph-structured pairwise interactions were further studied by Pass and Vargas-Jim\'enez \cite{PassVargasJimenez2023}. We also refer to the survey of Pass for a concise account of the theory and its applications \cite{Pass2015}. Some important examples include matching problems for teams in economics \cite{CarlierEkeland2010,PassMatching2014} and multi-marginal Coulomb costs arising in density functional theory \cite{ButtazzoDePascaleGoriGiorgi2012,CotarFrieseckeKluppelberg2013}.

In this paper we use two aspects of this theory. The first one is that the barycenter formulation gives canonical Brenier maps, and therefore leads to a Monge--Amp\`ere system. The second one is that the support of an optimal multi-marginal plan carries a natural indefinite geometry. In Pass's local theory, the corresponding metric depends on the second derivatives of the cost. For the quadratic cost studied here, this metric becomes the constant pseudo-Euclidean form used below. Hence the affine-geometric volume problem leads naturally to a volume estimate for spacelike graphs. Kolesnikov and Werner formulated the following conjecture.

\begin{kwconjecture}[Kolesnikov--Werner \cite{KW}]\label{conj1}
Let \(n\ge1\), \(k\ge2\). Let \(\rho:\R\to\R_+\) be non-increasing, and let \(f_1,\ldots,f_k:\R^n\to\R_+\) be even integrable functions.  Assume that
\begin{equation}\label{functional polarity}
 \prod_{i=1}^k f_i(x_i)
 \le
 \rho\bigl(\S_2(x_1,\ldots,x_k)\bigr),
 \qquad \forall\,x_1,\ldots,x_k\in\R^n.
\end{equation}
Then
\begin{equation}\label{func ineq part}
 \prod_{i=1}^k\int_{\R^n}f_i(x)\dd x
 \le
 \left(
 \int_{\R^n}\rho\bigl(C_k|u|^2\bigr)^{1/k}\dd u
 \right)^k .
\end{equation}
\end{kwconjecture}

For \(k=2\), this is the two-function functional Blaschke--Santal\'o inequality; see  \cite{artstein2004,ball1986,FM,lehec2009}.  For \(k\geq2\), Kolesnikov and Werner proved the conjecture when all functions are unconditional \cite{KW}.  Kalantzopoulos and Saroglou proved that it is enough to assume unconditionality of \(f_3,\ldots,f_k\) and related the functional statement to a geometric polarity condition \cite{KS}.  In the Gaussian case
\[
 \rho(t)=\exp\left(-\frac{t}{k-1}\right),
\]
Nakamura and Tsuji proved the conjecture for arbitrary even functions \cite{NT}. Courtade and Wang later gave an entropic proof of the Gaussian inequality and removed the symmetry assumptions in that setting \cite{CortWang2025}. The second-moment formulation is also closely related to Ball's strengthened Santal\'o conjecture, recently proved by B\"or\"oczky, Patsalos, and Saroglou \cite{BPS2026}.

Kalantzopoulos and Saroglou formulated the following geometric conjecture.

\begin{kwconjecture}[Kalantzopoulos--Saroglou \cite{KS}]\label{conj2}
Let \(n\ge1\), \(k\ge2\), and let \(K_1,\ldots,K_k\subset\R^n\) be origin-symmetric convex bodies.  If
\[
 \S_2(x_1,\ldots,x_k)\le C_k,
 \qquad \forall\,x_i\in K_i,\,i=1,\ldots,k,
\]
then
\[
 \prod_{i=1}^k |K_i|\le |B^n|^k .
\]
\end{kwconjecture}

Later, Kalantzopoulos and Saroglou showed that the functional assertion is equivalent to Conjecture~\ref{conj1} by using a level-set argument together with the multiplicative Pr\'ekopa--Leindler inequality; see \cite[Proposition~4.1]{KS}. We shall prove the geometric inequality in a slightly stronger form, without assuming convexity.

\begin{theorem}[Many-body Blaschke-Santal\'o inequality]\label{mainthm}
Let \(n\ge1\), \(k\ge2\), and let \(K_1,\ldots,K_k\subset\R^n\) be origin-symmetric measurable sets of finite volume.  If
\begin{equation}\label{constraint}
 \S_2(x_1,\ldots,x_k)\le C_k,
 \qquad \forall\,x_i\in K_i,
\end{equation}
then
\begin{equation}\label{main bound}
 \prod_{i=1}^k |K_i|\le |B^n|^k .
\end{equation}
If equality holds, then all \(|K_i|\) are positive and, up to null sets, exactly one of the following alternatives occurs:
\begin{enumerate}[label=\textup{(\roman*)}]
\item if \(k=2\), there exists \(A\in{\rm GL}(n)\) such that
\[
 K_1=AB^n,
 \qquad
 K_2=A^{-\top}B^n;
\]
\item if \(k\ge3\), then
\[
 K_1=\cdots=K_k=B^n .
\]
\end{enumerate}
Conversely, each case satisfies \eqref{constraint} and attains equality in \eqref{main bound}.
\end{theorem}

This rigidity result shows a clear difference between the two-body case and the many-body case. When \(k=2\), the affine invariance of the classical Blaschke--Santal\'o inequality is preserved. When \(k\ge 3\), the pairwise constraints impose additional restrictions, so that, after the prescribed normalization, equality holds only for Euclidean balls.

By the equivalence of Kalantzopoulos and Saroglou, Theorem~\ref{mainthm} therefore gives the full functional inequality. For completeness, we also record the equality statement.

\begin{theorem}[Functional many-body Blaschke-Santal\'o inequality]\label{thm:funct}
Let \(n\ge1\), \(k\ge2\). Let \(\rho:\R\to\R_+\) be non-increasing, and let \(f_1,\ldots,f_k:\R^n\to\R_+\) be even integrable functions satisfying \eqref{functional polarity}. Then \eqref{func ineq part} holds.

Within this admissible class, equality is characterized as follows.
\begin{enumerate}[label=\textup{(\roman*)}]
\item If \(k=2\), there exist constants \(c_1,c_2>0\) with \(c_1c_2=1\), and a symmetric positive-definite $n\times n$ matrix
\(Q\), such that
\[
    f_1(x)
    =
    c_1(\det Q)^{1/2}
    \rho\!\left(x^{\top}Qx\right)^{1/2},\quad
    f_2(x)
    =
    c_2(\det Q)^{-1/2}
    \rho\!\left(x^{\top }Q^{-1}x\right)^{1/2}, \quad\mathrm{for \,\,a.e.}\,x.
\]
\item If \(k\geq 3\), there exist constants
\(c_1,\ldots,c_k>0\) with $ \prod_{i=1}^{k}c_i=1$
such that
\[
    f_i(x)
    =
    c_i\,\rho\!\left(C_k|x|^2\right)^{1/k},
    \qquad i=1,\ldots,k,\quad\mathrm{for \,\,a.e.}\,x.
\]
\end{enumerate}

In either case, the constants are necessarily given by $
    c_i
    =
    \frac{\displaystyle\int_{\mathbb{R}^{n}}f_i(x)\,\dd x}
         {\displaystyle\int_{\mathbb{R}^{n}}
          \rho(C_k|x|^2)^{1/k}\,\dd x}.$
\end{theorem}

We outline the proof of Theorem~\ref{mainthm}. A convexification argument first reduces the problem to origin-symmetric convex bodies. Let $\mu_i=\frac{\1_{K_i}}{|K_i|}\dd x$
and let \(\nu\) be the equal-weight Wasserstein barycenter of \(\mu_1,\ldots,\mu_k\).  If \(T_i=\nabla\varphi_i\) denotes the Brenier map from \(\nu\) to \(\mu_i\), then $\gamma=(T_1,\ldots,T_k)_\#\nu$
is optimal for the quadratic multi-marginal problem
\[
 \sup_{\eta\in\Pi(\mu_1,\ldots,\mu_k)}
 \int \S_2(x_1,\ldots,x_k)\dd\eta .
\]
The support \(\Gamma=\supp\gamma\) has a hidden pseudo-Euclidean structure.  On \((\R^n)^k\),  let  $\mathcal B$ be the bilinear form defined in \eqref{bilinearform}. Then \(\mathcal B\) has signature \((n,n(k-1))\) and \(\mathcal B(X,X)=2\S_2(X)\). Kantorovich duality and cyclical monotonicity give
\[
 \mathcal B(X-Y,X-Y)\ge0,
 \qquad \forall\,X,Y\in\Gamma.
\]
After a coordinate-exchange argument and a linear pseudo-isometry, $\Gamma$ is therefore the graph of an odd $1$-Lipschitz map $f:A\to\R^{n(k-1)}$ for some $A\subset \R^n$. The polarity condition gives
\begin{equation}\label{pseuball}
 0\le |x|^2-|f(x)|^2\le 2C_k
 \qquad \forall\,x\in A.
\end{equation}

The new geometric ingredient is a pseudo-Euclidean volume estimate. If $A=-A\subset\R^n$ is bounded, $f:A\to\R^m$ is odd and $1$-Lipschitz, and satisfies \eqref{pseuball},
then
\begin{equation}\label{up vol esti}
 \Vol_{\mathcal B}(\Gamma)= \int_A\sqrt{\det(I-(\D f_x)^\top\D f_x)}\dd x
 \le  |B^n|(2C_k)^{n/2} .
\end{equation}
The proof uses the Dirichlet theory for higher-codimension maximal graphs of Li \cite{Li}, followed by a sharp monotonicity argument for maximal spacelike submanifolds. 

The lower estimate comes from the Monge--Amp\`ere equations for the barycenter maps. Define a projection
\[
 P:(\R^n)^k\to\R^n,\qquad P(X)=\sum_{i=1}^k x_i,
\] 
In the coordinate \(p=\sum_i x_i\), write
\[
 P^{-1}(p)=\big(W_1(p),\ldots,W_k(p)\big),
 \qquad
 M_i(p)=\D W_i(p).
\]
Then, almost everywhere, one has
\[
 M_i\ge0,
 \qquad
 \sum_{i=1}^k M_i=I,
 \qquad
 \det M_i=|K_i|h(p),
\]
where $h$ is the scaled density of the Wasserstein barycenter $\nu$. The induced metric on \(\Gamma\) is
 $G=I-\sum_{i=1}^k M_i^2.$
A sharp determinant inequality gives
\[
 \det(G)^{1/n}
 \ge 2C_k\left(\prod_{i=1}^k\det M_i\right)^{2/(kn)}.
\]
After integration,
\begin{equation*}
 \Vol_{\mathcal B}(\Gamma)
 \ge
 \bigl(2C_k\bigr)^{n/2}
 \left(\prod_{i=1}^k |K_i|\right)^{1/k}.
\end{equation*}
Combining this with \eqref{up vol esti}, we have the desired sharp many-body Blaschke-Santal\'o inequality.

The equality proof follows the same chain. For \(k=2\) it is the equality case of the symmetric Blaschke--Santal\'o inequality.  For \(k\ge3\), equality in the determinant inequality implies
\[
 M_1=\cdots=M_k=\frac1k I
\]
almost everywhere. The Monge--Ampère equations then show that the barycenter is of the form \(\dd\nu=\alpha 1_E \dd z\) for some origin-symmetric measurable set \(E\) and some constant $\alpha$. A non-expansion rigidity argument gives \(K_1=\cdots=K_k\). Substituting \(x_1=\cdots=x_k=x\) in \eqref{constraint} then gives that body to be $B^n$.

The paper is organized as follows. Section~\ref{Preliminaries} develops the pseudo-Euclidean geometry of the canonical multi-marginal support. Section~\ref{sec of bc and MA} proves the Monge--Amp\`ere lower bound. Section~\ref{pse vol estimate} establishes the pseudo-Euclidean volume estimate \eqref{up vol esti}. The geometric inequality is completed in Section~\ref{main proof of ineq}, and its equality cases are treated in Section~\ref{sec:geomeq}. Section~\ref{main proof of equality} contains the functional equality analysis.

\section{Preliminaries}\label{Preliminaries}
We denote by $B^n$ the Euclidean unit ball, and put $\R_+=(0,\infty)$. A convex body means a compact convex subset of $\mathbb R^n$ with non-empty interior. In this paper, symmetry always means symmetry with respect to the origin. For a set $A\subset\mathbb R^n$, we denote by $\conv A$, $\cl A$, and $\partial A$ its convex hull, closure, and boundary, respectively. If $f:A\subset\mathbb R^n\to\mathbb R^m$ is differentiable at $x$, then $\D f_x$ denotes its Jacobian matrix and $\|\D f_x\|$ denotes its operator norm. For other standard notation in convex geometry, we refer to \cite{Gardner,Schneider}.

We first make a simple convexification reduction. This step is purely convex-geometric. Namely, we replace the given sets by their closed convex hulls, while the many-body polar constraint remains unchanged.

\begin{lemma}\label{convexification}
Let $K_1,\ldots,K_k\subset\R^n$ be origin-symmetric measurable sets of positive finite volume satisfying \eqref{constraint}. Then $ \widetilde K_i=\cl \conv(K_i)$ is an origin-symmetric convex body for every $i$, and 
\[
\S_2(x_1,\ldots,x_k)\le C_k,
 \qquad \forall\,x_i\in \widetilde K_i.
\]
\end{lemma}

\begin{proof}

The function $\S_2$ is affine with respect to each variable. Hence, if we take convex hulls successively in each coordinate and then take closures, the constraint \eqref{constraint} is still preserved by continuity.

It is clear that each $\widetilde K_i$ has non-empty interior, since $|K_i|>0$. It remains to show that $\widetilde K_i$ is bounded. Fix $i$ and choose $j\ne i$. Since $\widetilde K_j$ has non-empty interior, there exists $\delta>0$ such that $ \delta B^n\subset \widetilde K_j.$
Moreover, all the sets contain the origin. Thus, by taking all variables except the $i$-th and the $j$-th ones to be zero, we get
\[
 |\ip{x}{y}|\le C_k,\qquad \forall\,x\in \widetilde K_i,\ y\in \widetilde K_j,
\]
where we have also used the symmetry of $\widetilde K_j$. For any $x\ne0$, taking $y=\delta x/|x|$ gives $\delta |x|\le C_k.$
Therefore, $|x|\le C_k/\delta$ for all $x\in \widetilde K_i$. This proves that $\widetilde K_i$ is bounded.
\end{proof}

It is therefore enough to prove Theorem \ref{mainthm} for origin-symmetric convex bodies, since $|K_i|\le|\widetilde K_i|$.
The link of the many-body inequality with quadratic multi-marginal transport is through the pairwise interaction term. For the equal-weight barycenter cost,
\[
        c(x_1,\ldots,x_k)
        =\min_{z\in\R^n}\sum_{i=1}^k |x_i-z|^2
        =\frac1k\sum_{1\le i<j\le k}|x_i-x_j|^2 .
\]
Since
\[
        \sum_{1\le i<j\le k}|x_i-x_j|^2
        =(k-1)\sum_{i=1}^k |x_i|^2
        -2\S_2(x_1,\ldots,x_k),
\]
minimizing this quadratic cost over couplings with fixed marginals is
equivalent to maximizing $\S_2$. Thus $\S_2$ is the relevant
interaction part of the quadratic Gangbo--\'{S}wi\k{e}ch cost. Its
polarization reveals the pseudo-Euclidean structure used below.
For $X=(x_1,\ldots,x_k),Y=(y_1,\ldots,y_k)\in(\R^n)^k,$ one can define
\begin{equation}\label{bilinearform again}
        \mathcal B(X,Y)
        =\big\langle\sum_{i=1}^k x_i,\sum_{i=1}^k y_i\big\rangle
        -\sum_{i=1}^k\ip{x_i}{y_i}.
\end{equation}
Then $\mathcal B(X,X)=2\S_2(X)$ and we have 
\begin{equation}\label{BXX bound}
        \mathcal B(X,X)\le 2C_k,
\end{equation}
where the upper bound is exactly the many-body polarity assumption \eqref{constraint}.

For later use, we recall some notation for graphs in a general pseudo-Euclidean space. We write the pseudo-Euclidean space of signature $(n,m)$ as
\[
        \R^{n,m}=\R_x^n\oplus\R_y^m,
        \qquad
        \big\langle(x,y), (x',y')\big\rangle_{n,m}
        =\ip{x}{x'}-\ip{y}{y'},
\]
and denote the corresponding metric by
\[
        g_{n,m}=\sum_{\alpha=1}^n \dd x_\alpha^2
        -\sum_{\beta=1}^m \dd y_\beta^2 .
\]
Thus, for $v=(x,y)\in\R^{n,m}$,
\[
        \ip{v}{v}_{n,m}=g_{n,m}(v,v)=|x|^2-|y|^2 .
\]
As usual, a vector $v\in\R^{n,m}$ is called spacelike, null, or timelike according as
\[
        \ip{v}{v}_{n,m}>0,\qquad
        \ip{v}{v}_{n,m}=0,\qquad
        \ip{v}{v}_{n,m}<0 .
\]
An $n$-dimensional submanifold $\Sigma\subset\R^{n,m}$ is called spacelike if the metric induced on each tangent space is positive definite. If, in addition, its mean curvature vector vanishes identically, then $\Sigma$ is called a maximal submanifold. For more background on these notions, we refer to \cite{One83}.

Let $A\subset\R^n$ and let $f:A\to\R^m$ be $1$-Lipschitz. By Rademacher's theorem \cite[Section 3.1.2]{EvansGariepy}, $\|\D f_x\| \le1$ for a.e. $x\in A$, so $I-(\D f_x)^\top\D f_x$ is positive semidefinite a.e.. We define the pseudo-volume of the weakly spacelike graph of $f$ by
\begin{equation}\label{graph volume}
        \Vol_{n,m}(\operatorname{graph}_A f)
        =\int_A\sqrt{\det\big(I-(\D f_x)^\top\D f_x\big)}\dd x .
\end{equation}
If $f$ is defined only on $A$, one may take any global Lipschitz extension; the value in \eqref{graph volume} is independent of the choice of extension, since any two Lipschitz extensions that agree on \(A\) have the same approximate
differential for a.e. \(x\in A\).  Indeed, for the parametrization $F(x)=(x,f(x))$, at a.e. \(x\in A\) where \(f\) is approximately
differentiable,
\[
        (F^*g_{n,m})_x(v,w)
        =\ip{v}{w}-\ip{\D f_xv}{\D f_xw},
\]
and the Gram matrix is $I-(\D f_x)^\top\D f_x$.

Hence, by the discussion in introduction, we know that \(\mathcal B\) is a pseudo-Euclidean metric on \((\mathbb R^n)^k\) with signature \((n,n(k-1))\). We denote by \(\Vol_{\mathcal B}\) the induced \(n\)-dimensional pseudo-volume.

We now pass to the multi-marginal optimal transport theory. The following theorem collects consequences of Agueh–Carlier's barycenter theory \cite{AC}, which will be used for the subsequent analysis. The existence of a minimizer follows from Agueh--Carlier
\cite[Proposition~2.3]{AC}.

\begin{theorem}\label{AC theorem}
Let \(K_1,\ldots,K_k\subset\R^n\) be origin-symmetric convex bodies and set
\[
 \mu_i=\frac{\1_{K_i}}{|K_i|}\dd x,\qquad 
 i=1,\ldots,k.
\]
Then the functional
\begin{equation*}
 \mathcal J(\eta):=\sum_{i=1}^k W_2^2(\eta,\mu_i)
\end{equation*}
admits a unique minimizer \(\nu\) among all probability measures with finite second moment. We call \(\nu\) the equal-weight
Wasserstein barycenter of \(\mu_1,\ldots,\mu_k\). Moreover, the following hold:

\begin{enumerate}[label=\textup{(\arabic*)}]
\item \cite[Proposition 3.5 and Theorem 5.1]{AC}. The barycenter \(\nu\) is unique, origin-symmetric, and absolutely continuous with bounded density,
$ \dd\nu(z)=\bar h (z)\dd z .$
\item For each \(i\), by Brenier-McCann's theorem (\cite{Brenier,McCann1995}), there exists a convex potential $\varphi_i$ such that
\[
 T_i=\nabla\varphi_i,\qquad (T_i)_\#\nu=\mu_i.
\]
\item At Lebesgue-a.e. differentiability point $z$ of $T_i$ with $\bar h(z)>0$,
\begin{equation}\label{MA equation}
 \det \D T_i(z)=|K_i|\bar h(z).
\end{equation}

\item \cite[Remark 3.9]{AC}. It holds that
\begin{equation}\label{bc condition}
 \frac1k\sum_{i=1}^k T_i(z)=z,
 \qquad \nu\text{-}\mathrm{a.e. }\,z.
\end{equation}

\item Each \(T_i\) is odd \(\nu\)-${\rm {a.e.}}$.

\item \cite[Theorem 4.1 and Proposition~4.2]{AC}. The plan
\begin{equation}\label{canonical plan}
 \gamma=(T_1,\ldots,T_k)_\#\nu
\end{equation}
is optimal for the multi-marginal problem
\begin{equation}\label{MMOT}
 \sup_{\eta\in\Pi(\mu_1,\ldots,\mu_k)}
 \int_{K_1\times\cdots\times K_k}
 \S_2(x_1,\ldots,x_k)\,\dd\eta .
\end{equation}
\end{enumerate}
\end{theorem}

\begin{proof}
The existence and uniqueness of the barycenter, and the \(L^\infty\) bound for its density in (1) and the first-order condition in (4), follow from Agueh--Carlier
\cite[Proposition~3.3, Proposition~3.5, Definition~3.9 and Theorem~5.1]{AC}. Applying the two-marginal Brenier-McCann's theorem to each pair \((\nu,\mu_i)\) gives a unique \(\nu\)-a.e. map
\(T_i=\nabla\varphi_i\) with \((T_i)_\#\nu=\mu_i\). The symmetry of $\nu$ follows from uniqueness, since each $\mu_i$ is invariant under $x\mapsto -x$.

For (3), the Monge--Amp\`ere equation for the Brenier map from $\nu$ to $\mu_i$ gives
\[
        \bar h(z)=\frac1{|K_i|}\det\D T_i(z)
\]
at a.e. differentiability point with $\bar h(z)>0$; see Brenier \cite{Brenier}.

To prove (5), set $\widetilde T_i(z)=-T_i(-z)$. The symmetries of $\nu$ and $\mu_i$ imply $(\widetilde T_i)_\#\nu=\mu_i$, and $\widetilde T_i$ is again the gradient of a convex function. By uniqueness of the Brenier map, $\widetilde T_i=T_i$ $\nu$-a.e., hence $T_i(-z)=-T_i(z)$ for $\nu$-a.e. $z$.

It remains to prove (6). We apply Agueh--Carlier's results with
\[
        p=k,\qquad d=n,\qquad \nu_i=\mu_i,\qquad \lambda_i=\frac1k,\qquad T(X)=\frac1k\sum_{i=1}^k x_i .
\]
Since each \(\mu_i\) is absolutely continuous, the hypotheses of
\cite[Theorem~4.1 and Proposition~4.2]{AC} are satisfied. Let
\(\widehat\gamma\in\Pi(\mu_1,\ldots,\mu_k)\) denote the unique optimizer of Agueh--Carlier's problem, it suffices to identify \(\widehat\gamma\) with the canonical plan \((T_1,\ldots,T_k)_\#\nu\).

In the equal-weight case, their barycenter problem is
\[
        \inf_\sigma \frac1{2k}\sum_{i=1}^k W_2^2(\mu_i,\sigma),
\]
which has the same minimizers as \eqref{MMOT}. Hence, by
the uniqueness of the barycenter and by \cite[Proposition~4.2]{AC},
\begin{equation*}
        T_\#\widehat\gamma=\nu .
\end{equation*}

 Let \(\pi_i\) be the \(i\)-th coordinate projection.
Since \((T,\pi_i)_\#\widehat\gamma\in\Pi(\nu,\mu_i)\), we have
\begin{equation}\label{OT coupling}
        W_2^2(\nu,\mu_i)
        \le
        \int_{K_1\times\cdots\times K_k} |T(X)-x_i|^2\,\dd\widehat\gamma(X),
        \qquad X=(x_1,\ldots,x_k).
\end{equation}
On the other hand, the comparison argument in the proof of
\cite[Proposition~4.2]{AC} gives the following inequality: for every
probability measure \(\sigma\) and every choice of couplings
\(\alpha_i\in\Pi(\mu_i,\sigma)\),
\[
        \frac1{2k}\sum_{i=1}^k
        \int_{K_i\times\R^n} |x_i-y|^2\,\dd\alpha_i(x_i,y)
        \ge
        \int_{K_1\times\cdots\times K_k} \frac1{2k}\sum_{i=1}^k |x_i-T(X)|^2\,\dd\widehat\gamma(X).
\]
Taking \(\sigma=\nu\) and choosing \(\alpha_i\) to be optimal couplings between
\(\mu_i\) and \(\nu\), we obtain
\[
        \frac1{2k}\sum_{i=1}^k W_2^2(\mu_i,\nu)
        \ge
        \int_{K_1\times\cdots\times K_k} \frac1{2k}\sum_{i=1}^k |x_i-T(X)|^2\,\dd\widehat\gamma(X).
\]
Combining this  with \eqref{OT coupling}, we have that the equality in
\eqref{OT coupling} holds for every \(i\). Therefore
   $(T,\pi_i)_\#\widehat\gamma$ is an optimal quadratic coupling from \(\nu\) to \(\mu_i\).

 Since \(\nu\) is absolutely continuous,  by Brenier's theorem, we have the uniqueness of the optimal coupling from \(\nu\) to \(\mu_i\).  Hence,   $(T,\pi_i)_\#\widehat\gamma         =
        (\mathrm{id},T_i)_\#\nu ,$  which implies that for every bounded Borel
function \(\varphi\),
 \[
 \begin{aligned}
  \int_{K_1\times\cdots\times K_k} \varphi(x_1,\ldots,x_k)\,\dd\widehat\gamma(X)=
 \int_{\R^n} \varphi\bigl(T_1(z),\ldots,T_k(z)\bigr)\,\dd\nu(z).
 \end{aligned}
 \]
Thus $ \widehat\gamma=(T_1,\ldots,T_k)_\#\nu .$

\end{proof}

We shall use only the support of the optimal plan constructed above.  We recall the standard Kantorovich duality theorem in a form adapted to $\S_2$.

\begin{lemma}[Kantorovich duality]\label{Kantorovich duality}
Let \(\gamma\) be an optimizer for \eqref{MMOT}. Then Kantorovich duality gives
\[
\sup_{\eta\in\Pi(\mu_1,\ldots,\mu_k)}
\int_{K_1\times\cdots\times K_k} \S_2\,\dd\eta
=
\inf_{\substack{\phi_i\in C(K_i)\\
\sum_{i=1}^k\phi_i(x_i)\ge \S_2(x_1,\ldots,x_k)}}
\sum_{i=1}^k\int_{K_i}\phi_i\,\dd\mu_i .
\]
Moreover, the dual infimum is attained. Thus there exist
\(\phi_i\in C(K_i)\), \(i=1,\ldots,k\), such that
\begin{equation}\label{dual ineq}
 \sum_{i=1}^k\phi_i(x_i)\ge \S_2(X),
 \qquad\forall\,X=(x_1,\ldots,x_k)\in K_1\times\cdots\times K_k,
\end{equation}
and
\begin{equation}\label{dual contact}
 \sum_{i=1}^k\phi_i(x_i)=\S_2(X),
 \qquad \gamma\text{-}\mathrm{a.e.}\,X.
\end{equation}
If $ \Gamma=\supp\gamma,$ then \(\Gamma\) is compact and
\begin{equation}\label{dual contact support}
 \sum_{i=1}^k\phi_i(x_i)=\S_2(X),
 \qquad\forall\, X\in\Gamma.
\end{equation}
If $\gamma$ is the barycentric optimal plan \eqref{canonical plan}, then $\Gamma$ is origin-symmetric.
\end{lemma}

\begin{proof}
Since  \(K_1,\ldots,K_k\) are compact and \(\S_2\) is continuous, compact
multi-marginal Kantorovich duality applies; see Kellerer
\cite{Kellerer84}, and also the
standard optimal-transport references
\cite{Santambrogio15, Villani09}. In the quadratic
case, the same dual and contact formulation appears in Gangbo--Święch
\cite{GS} and Agueh--Carlier
\cite{AC}. This gives the dual formula and dual attainment.

% Since $\sum_{i=1}^k\phi_i(x_i)\geq \S_2(X)$ on \(K_1\times\cdots\times K_k\), and $\sum_{i=1}^k\phi_i(x_i)= \S_2(X)$ for $\gamma\text{-}\mathrm{a.e.}\,X$, 
Assume that there is \(X_0=(x_{0,1},\ldots,x_{0,k})\in\supp\gamma\) such that $   \sum_{i=1}^k\phi_i(x_{0,i})> \S_2(X_0).$
By the continuity of \(\phi_1,\ldots,\phi_k\) and \(\S_2\), this strict
inequality also holds in a neighborhood of \(X_0\), which has positive
\(\gamma\)-measure, contradicting
\eqref{dual contact}. Thus $   \sum_{i=1}^k\phi_i(x_i)= \S_2(X)$
on \(\Gamma\).

Finally, if \(\gamma=(T_1,\ldots,T_k)_\#\nu\), then \(\nu\) is symmetric and the maps \(T_i\) are odd \(\nu\)-a.e.. Therefore \(\gamma\) is invariant under
\(X\mapsto -X\), and hence \(\Gamma=-\Gamma\).
\end{proof}

The following lemma is standard, which is a two-point coordinate-exchange consequence of the corresponding multi-marginal \(c\)-cyclical monotonicity.
\begin{lemma}\label{monoto results}
Let \(\Gamma\) be the support set obtained in Lemma
\ref{Kantorovich duality}. If
\[
 X=(x_1,\ldots,x_k),\qquad Y=(y_1,\ldots,y_k)
\]
belong to \(\Gamma\), then for every \(i=1,\ldots,k\),
\begin{equation}\label{monotonicity}
 \big\langle x_i-y_i,\sum_{j\ne i}(x_j-y_j)\big\rangle\ge0.
\end{equation}
Consequently,
\begin{equation}\label{monotonicity 1}
 \left|\sum_{i=1}^k(x_i-y_i)\right|^2
 -\sum_{i=1}^k|x_i-y_i|^2\ge0,
\end{equation}
or equivalently
\begin{equation}\label{B monotonicity}
 \mathcal B(X-Y,X-Y)\ge0 .
\end{equation}
\end{lemma}

\begin{proof}
Fix $1\le i \le k$. Let $X'$ and $Y'$ be obtained from $X$ and $Y$ by exchanging only their $i$-th coordinates:
\[
        X'=(x_1,\ldots,y_i,\ldots,x_k),
\]
\[
        Y'=(y_1,\ldots,x_i,\ldots,y_k).
\]
Using \eqref{dual contact support} at $X,Y$ and \eqref{dual ineq} at $X',Y'$, we get
\[
        \S_2(X)+\S_2(Y)
        =\sum_{\ell=1}^k\phi_\ell(x_\ell)+\sum_{\ell=1}^k\phi_\ell(y_\ell)
        \ge \S_2(X')+\S_2(Y').
\]
Only the pairs involving the $i$-th coordinate change, and therefore
\[
        \S_2(X)+\S_2(Y)-\S_2(X')-\S_2(Y')
        =\big\langle x_i-y_i,\sum_{j\ne i}(x_j-y_j)\big\rangle .
\]
This proves \eqref{monotonicity}.  Summing \eqref{monotonicity} over $i$ gives
 \eqref{monotonicity 1}, which is equivalent to \eqref{B monotonicity}.
\end{proof}

Once we have \eqref{B monotonicity}, we can prove that the support of the optimal plan is the graph of an odd $1$-Lipschitz map. Indeed, we shall use the following barycentric decomposition of $(\R^n)^k$.
Let
\[
        \mathcal H
        =
        \left\{(v_1,\ldots,v_k)\in(\R^n)^k:\,
        \sum_{i=1}^k v_i=0\right\}.
\]
For \(X=(x_1,\ldots,x_k)\in(\R^n)^k\), set
\begin{equation*}
        p=\sum_{i=1}^k x_i,\qquad
        e=\sqrt{\frac{k-1}{k}}\,p,\qquad
        w_i=x_i-\frac pk.
\end{equation*}
Note that $w=(w_1,\ldots,w_k)\in\mathcal H.$ This defines a linear isomorphism
\[
        \Phi:(\R^n)^k\longrightarrow \R^n\oplus\mathcal H,
        \qquad
        \Phi(X)=(e,w),
\]
with inverse
\[
        x_i=w_i+\frac{1}{\sqrt{k(k-1)}}\,e,
        \qquad i=1,\ldots,k .
\]
If \(Y=(y_1,\ldots,y_k)\) is written analogously as \(\Phi(Y)=(e',w')\), then
\begin{equation}\label{new coordinate}
        \mathcal B(X,Y)
        =
        \ip{e}{e'}-\sum_{i=1}^k\ip{w_i}{w_i'}
        =
        \ip{e}{e'}-\ip{w}{w'} .
\end{equation}
Thus \(\Phi\) identifies \(((\R^n)^k,\mathcal B)\) with the pseudo-Euclidean space \(\R^n\oplus\mathcal H\), equipped with the quadratic form $ (e,w)\longmapsto |e|^2-|w|^2.$ 

We now apply this decomposition to the support of the optimal plan.

\begin{proposition}\label{P projection}
Let $\gamma=(T_1,\ldots,T_k)_\#\nu$ be the optimal plan obtained in
Theorem~\ref{AC theorem}. Denote $\Gamma=\supp\gamma$. Then the projection
\[
        P:(\R^n)^k\to\R^n,
        \qquad
        P(X)=\sum_{i=1}^k x_i,
\]
is injective on \(\Gamma\). Hence \(\Phi(\Gamma)=\{(e,f(e)):e\in A\}\) is the graph of a map \(f:A\to\mathcal H\), where \(A=P(\Gamma)\subset\mathbb R^n\).
Moreover, \(A\) is compact and origin-symmetric, and \(f\) is odd and
\(1\)-Lipschitz.

\end{proposition}

\begin{proof}
By Lemma~\ref{monoto results},
\[
        \mathcal B(X-Y,X-Y)\ge0,
        \qquad \forall\,X,Y\in\Gamma .
\]
Writing \(\Phi(X)=(e,w)\) and \(\Phi(Y)=(e',w')\), and using
\eqref{new coordinate}, this becomes
\[
        |e-e'|^2-|w-w'|^2\ge0 .
\]
In particular, if \(e=e'\), then \(w=w'\), hence \(X=Y\).  Since
\(e=\sqrt{(k-1)/k}\,P(X)\), this is equivalent to the injectivity of \(P\)
on \(\Gamma\).

Let $  A=P(\Gamma).$ The preceding injectivity shows that \(\Phi(\Gamma)\) is the graph of a
map \(f:A\to\mathcal H\).  Since \(\Gamma\) is compact and \(\Phi\) is
continuous, \(A\) is compact.  Moreover, for any  \(e,e'\in A\),
\[
        \big|f(e)-f(e')\big|^2\le |e-e'|^2 ,
\]
so \(f\) is \(1\)-Lipschitz.

The optimal plan \(\gamma\) is invariant under \(X\mapsto -X\), because the barycenter \(\nu\) is origin-symmetric and the Brenier maps \(T_i\) are odd \(\nu\)-a.e.. Hence \(\Gamma=-\Gamma\). Since \(\Phi\) is linear, this implies \(\Phi(\Gamma)=-\Phi(\Gamma)\). Therefore \(A=-A\), and the graph representation gives $f(-e)=-f(e)$ for any $ e\in A.$
\end{proof}

\section{Wasserstein barycenters and a Monge--Amp\`ere lower bound}\label{sec of bc and MA}

We keep the notation of Section~\ref{Preliminaries}. Recall that \(\mu_i=|K_i|^{-1}{\bf 1}_{K_i}\dd x\), \(\dd \nu=\bar h\dd z\) is the equal-weight Wasserstein barycenter of the \(\mu_1,\ldots,\mu_k\),
\(T_i=\nabla\varphi_i\) is the Brenier map from \(\nu\) to \(\mu_i\), and \(\gamma=(T_1,\ldots,T_k)_\#\nu\) is the corresponding quadratic multi-marginal optimizer. Let \(\Gamma=\supp\gamma\). Recall $P:(\R^n)^k\to\R^n$ is defined in Proposition \ref{P projection} as $P(X)=\sum_i x_i$. Let $A=P(\Gamma)$. 
By \eqref{monotonicity 1}, the inverse map $F=P^{-1}:A\to\Gamma$ is $1$-Lipschitz in the  Euclidean norm. Write
\[
 F(p)=\bigl(W_1(p),\ldots,W_k(p)\bigr).
\]
Then
\begin{equation}\label{F 1-lip}
 \sum_{i=1}^k\big|W_i(p)-W_i(q)\big|^2\le|p-q|^2,
 \qquad
 \sum_{i=1}^kW_i(p)=p.
\end{equation}
By \eqref{bc condition}, \eqref{canonical plan} and \eqref{F 1-lip} we have 
\begin{equation*}
 W_i(p)=T_i(p/k),
\qquad\text{for a.e. }p\text{ with respect to }h(p)\dd p,
\end{equation*}
where \begin{equation}\label{desity}
h (p)=k^{-n}\bar h(p/k),
 \qquad \int_{\R^n}h(p)\dd p=1.
\end{equation}

Let $E=A\cap\{p:h(p)>0\}$. For $\text{a.e. }p\in E$,  we have
\begin{equation*}
 M_i(p):=\D W_i(p)=\frac1k\D T_i(p/k),
 \qquad
 M_i=M_i^\top\geq0,
 \qquad
 \sum_{i=1}^kM_i=I.
\end{equation*}
In particular, $M_i\geq 0$ and $I-M_i\geq0$.  By \eqref{MA equation},
\begin{equation}\label{MAE}
 \det M_i(p)=|K_i| h(p),
 \qquad\text{for a.e. }p\in E.
\end{equation}
With respect to the bilinear form \(\mathcal B\) in \eqref{bilinearform again}, the pullback metric induced by \(F\) satisfies, for \(\xi_1,\xi_2\in T_pA\simeq \mathbb{R}^n\),
\begin{align*}
(F^*\mathcal B)_p(\xi_1,\xi_2)
=
\langle \xi_1,\xi_2\rangle
-
\sum_{i=1}^k
\big\langle M_i(p)\xi_1,M_i(p)\xi_2\big\rangle
=
\xi_1^\top
G(p)
\xi_2,
\end{align*}
where $G(p)=I-\sum_{i=1}^kM^2_i(p)$. Since $M_i-M_i^2\geq 0$, one has $G\geq0$. By the area formula, we have
%or equivalently the change from the $p$-coordinate to the graph coordinate $e=\sqrt{(k-1)/k}\,p$, 
\begin{equation}\label{vol formula}
 \Vol_{\mathcal B}(\Gamma)
 =\int_{A}\sqrt{\det G(p)}\dd p
 \ge\int_E\sqrt{\det G(p)}\dd p.
\end{equation}
The required pointwise estimate follows from a direct application of Minkowski's determinant inequality.
\begin{lemma}\label{det lemma}
Let $M_1,\ldots,M_k$ be positive semidefinite $n\times n$ matrices satisfying $\sum_iM_i=I$. Let $G=I-\sum_iM_i^2$. Then
\begin{equation}\label{det ineq}
 \det(G)^{1/n}
 \ge k(k-1)
 \left(\prod_{i=1}^k\det M_i\right)^{2/(kn)}.
\end{equation}
For $k=2$, equality is an identity.  If $k\ge3$ and all $M_i$ are positive definite, equality holds if and only if
\begin{equation}\label{matrix rigidity}
 M_1=\cdots=M_k=\frac1k I.
\end{equation}
\end{lemma}

\begin{proof}
By the assumptions, it is easy to see that  $M_i(I-M_i)$ is  positive semidefinite, and
\[
 G=\sum_{i=1}^kM_i(I-M_i).
\]
It follows from Minkowski's determinant inequality that
\begin{align*}
 \det(G)^{1/n}
 \ge\sum_{i=1}^k\det\bigl(M_i(I-M_i)\bigr)^{1/n}=\sum_{i=1}^k (\det M_i)^{1/n}\det(I-M_i)^{1/n}.
\end{align*}
Since $I-M_i=\sum_{j\ne i}M_j$, by Minkowski's inequality again, we have
\[
 \det(I-M_i)^{1/n}\ge\sum_{j\ne i}(\det M_j)^{1/n}.
\]
Consequently,
\[
 \det(G)^{1/n}\ge2\sum_{i<j}(\det M_i)^{1/n}(\det M_j)^{1/n}\geq k(k-1) \left(\prod_{i=1}^k(\det M_i)^{1/n}\right)^{2/k},
\]
 which proves \eqref{det ineq}.

If $k=2$, a direct computation shows that \eqref{det ineq} is an identity. Assume $k\ge3$ and all $M_i$ are positive definite.  Equality in the last inequality implies $(\det M_1)^{1/n}=\cdots=(\det M_k)^{1/n}$. Equality in the Minkowski inequality implies that all $M_i$ are proportional.  Hence the matrices are equal, and $\sum_iM_i=I$ gives \eqref{matrix rigidity}.  

\end{proof}

Now we can give a lower bound of the pseudo-volume of $\Gamma.$
\begin{proposition}\label{lower vol}
For the optimal graph $\Gamma$ associated with $K_1,\ldots,K_k$,
\begin{equation}\label{lower volume}
 \Vol_{\mathcal B}(\Gamma)
 \ge\big(k(k-1)\big)^{n/2}
 \left(\prod_{i=1}^k|K_i|\right)^{1/k}.
\end{equation}
\end{proposition}

\begin{proof}
By \eqref{det ineq} and \eqref{MAE}, for a.e. \(p\in E\), we have
\begin{align*}
 \sqrt{\det G(p)}
 \ge\big(k(k-1)\big)^{n/2}
 \left(\prod_i\det M_i(p)\right)^{1/k}=\big(k(k-1)\big)^{n/2}
 \left(\prod_i|K_i|\right)^{1/k} h(p).
\end{align*}
Integrating over $E$, and then using \eqref{desity} and \eqref{vol formula}, we have the desired estimate \eqref{lower volume}.
\end{proof}

\section{A pseudo-Euclidean volume estimate}\label{pse vol estimate}

In this section, we prove the upper estimate for $\Vol_{\mathcal B}(\Gamma)$. By Proposition \ref{P projection}, the upper bound has been reduced to a purely pseudo-Euclidean volume comparison. Indeed, since \(\Phi(\Gamma)=\{(e,f(e)):e\in A\}\), we have
\[
        \Vol_{\mathcal B}(\Gamma)
        =
        \int_A
        \sqrt{\det\bigl(I-(\D f_x)^\top \D f_x\bigr)}\,\dd x .
\]
Moreover,
\[
        0\le |x|^2-|f(x)|^2\le 2C_k.
\]
Thus \eqref{up vol esti} follows from the following lemma, applied with \(R=(2C_k)^{1/2}\).

\begin{lemma}\label{upper bound of 1-lip}
Let \(R>0\). Let \(A=-A\subset\R^n\) be bounded and measurable, and let \(f:A\to\R^m\) be odd and \(1\)-Lipschitz. Assume that
\begin{equation}\label{pseudoball}
 0\le |x|^2-|f(x)|^2\le R^2,
 \qquad \forall\,x\in A.
\end{equation}
Then
\begin{equation}\label{upperbd of 1-lip}
 \int_A\sqrt{\det\bigl(I-(\D f_x)^\top \D f_x\bigr)}\dd x
 \le |B^n|R^n.
\end{equation}
The estimate is sharp on every spacelike $n$-plane through the origin.
\end{lemma}
To prove Lemma \ref{upper bound of 1-lip}, we first record the maximal graph Dirichlet theorem from \cite{Li}, which will be used below.

\begin{lemma}[{\cite[Theorems~1.1, 2.1, and Lemma~2.2]{Li}}]\label{Li theorem}
Let $\Omega\subset\R^n$ be a bounded domain with smooth boundary, and
let $\varphi:\partial\Omega\to\R^m$ be smooth and for some $q<1$,
\[
 \big|\varphi(x)-\varphi(y)\big|\leq q|x-y|,
 \qquad \forall\,x,y\in\partial\Omega,\quad x\ne y.
\]
Then the following statements hold.
\begin{enumerate}[label=\textup{(\arabic*)}]
 \item There exists a unique smooth map
 \[
  u:\overline{\Omega}\to\R^m,
  \qquad u|_{\partial\Omega}=\varphi,
 \]
 whose graph
 \[
  \Sigma_u=\big\{(x,u(x)):x\in\Omega\big\}\subset\R^{n,m}
 \]
 is spacelike and maximal.

 \item The graph $\Sigma_u$ maximizes the pseudo-Riemannian volume 
 among all smooth spacelike graphs with the same Dirichlet boundary
 data. 
 
 \item The map $u$ is globally acausal:
\[
 \big|u(x)-u(y)\big|<|x-y|,
 \qquad \forall\,x,y\in\Omega,\quad x\ne y.
\]
\end{enumerate}
\end{lemma}

We now prove the volume estimate.

\begin{proposition}[Smooth graph estimate]\label{Smooth graph estimate} 
Let \(R>0\). Let \(g:\R^n\to\R^m\) be smooth, odd, and \(q\)-Lipschitz for some \(q<1\). If \(A\subset\R^n\) is bounded and
\[
 |x|^2-|g(x)|^2\le R^2, \qquad \forall\,x\in A.
\]
Then
\begin{equation}\label{upbd s1lip}
 \int_A\sqrt{\det\big(I-(\D g_x)^\top \D g_x\big)}\dd x\le|B^n|R^n.
\end{equation}
\end{proposition}

\begin{proof}
Denote $\mathcal{R}(x)=|x|^2-|g(x)|^2$. Since \(g\) is odd, \(g(0)=0\), and therefore \(|g(x)|\le q|x|\). For \(x\ne0\) and \(t>0\),
\begin{align}\label{diff g}
 \frac{\dd}{\dd t}\mathcal{R}(tx)
 =2t|x|^2-2g(tx)\cdot \D g(tx)x\ge2(1-q^2)t|x|^2>0.
\end{align}
In particular, $\mathcal{R}(x)\ge(1-q^2)|x|^2$.

Fix \(R_0>R\), and define
\[
        E=\big\{x\in\R^n:\mathcal{R}(x)<R_0^2\big\}.
\]
The estimate \(\mathcal{R}(x)\ge(1-q^2)|x|^2\) shows that \(E\) is bounded, and the assumption on \(A\) implies \(A\subset E\). Since \(\mathcal{R}\) is even,  \(E\) is symmetric with respect to the origin. If \(x\in\partial E\), then \(x\ne0\), and \eqref{diff g} give
\[
        \nabla \mathcal{R}(x)\cdot x
        =
        \left.\frac{\dd}{\dd t}\mathcal{R}(tx)\right|_{t=1}
        >0 .
\]
Hence \(E\) is a smooth symmetric bounded domain. On \(\partial E\),
\begin{equation}\label{R0 ball}
        |x|^2-|g(x)|^2=R_0^2 .
\end{equation}
Since the boundary value \(g|_{\partial E}\) is smooth and acausal, Lemma~\ref{Li theorem} gives a unique smooth maximal spacelike graph
\[
 \Sigma=\bigl\{F(x)=(x,u(x)):x\in E\bigr\}
\]
with \(u=g\) on \(\partial E\), and this graph maximizes the pseudo-Riemannian volume among smooth spacelike graphs with the same boundary values. Hence
\begin{equation}\label{max vol}
         \int_E\sqrt{\det\bigl(I-(\D g_x)^\top\D g_x\bigr)}\dd x
        \le \Vol_{n,m}(\Sigma).
\end{equation}
	It remains to bound the volume of \(\Sigma\). By uniqueness, \(u\) is odd and hence \(u(0)=0\). By Lemma~\ref{Li theorem}, we have
	\[
	|u(x)|=|u(x)-u(0)|<|x|,
	\qquad \forall\,x\in E,\ x\ne0 .
	\]
	Thus \(\|F\|_{n,m}^2:=\langle F,F\rangle_{n,m}=|x|^2-|u(x)|^2\) is positive on \(\Sigma\setminus\{0\}\). Since $g=u$ on \(\partial\Sigma\), and by \eqref{R0 ball} we have
	\(\|F\|_{n,m}=R_0\) on \(\partial\Sigma\).

    We claim that 
\begin{equation}\label{claim}
    \Vol_{n,m}(\Sigma)\leq |B^n|R_0^n.
\end{equation}    
	Since \(\Sigma\) is maximal, \(\Delta_\Sigma F=0\). Therefore, for any local orthonormal tangent frame \(e_1,\ldots,e_n\),
	\begin{equation}\label{laplacian bound}
		\Delta_\Sigma\frac{\|F\|_{n,m}^2}{2}
		=
		\langle \Delta_\Sigma F,F\rangle_{n,m}
		+
		\sum_{\alpha=1}^n
		\big\langle \dd F(e_\alpha),\dd F(e_\alpha)\big\rangle_{n,m}
		=n .
	\end{equation}
Decompose $F$ into its tangential and normal components:
\[
F=F^\parallel+F^\perp.
\]
Since the orthogonal complement of a spacelike $n$-plane in signature $(n, m)$ is negative definite,
\[
\langle F^\perp,F^\perp\rangle_{n,m}\le 0.
\]
Moreover, since $\nabla_\Sigma\frac{\|F\|_{n,m}^2}{2}=F^\parallel,$ we have that on \(\Sigma\setminus\{0\}\),
	\begin{equation}\label{decp}
		\bigl|\nabla_\Sigma \|F\|_{n,m}\bigr|^2
		=
		\frac{\langle F^\parallel,F^\parallel\rangle_{n,m}}{\|F\|_{n,m}^2}
		=
		1-\frac{\langle F^\perp,F^\perp\rangle_{n,m}}{\|F\|_{n,m}^2}
		\ge1 .
	\end{equation}
We next apply this radial estimate to the sublevel sets of
\(\|F\|_{n,m}\). For \(0<s<R_0\), define
\[
        V(s):=\Vol_{n,m}\bigl\{p\in\Sigma:\ \|F(p)\|_{n,m}<s\bigr\}.
\]
By Sard's theorem,  almost every
\(s\in(0,R_0)\) is a regular value of \(\|F\|_{n,m}\). For such \(s\),
\[
        \big\{\|F\|_{n,m}=s\big\}\subset \Sigma 
\] 
is a smooth hypersurface in \(\Sigma\), and we denote by \(\dd A_s\) its induced \((n-1)\)-dimensional area element.

For every regular \(s\in(0,R_0)\), the divergence theorem applied to \eqref{laplacian bound} on the sublevel set \(\{\|F\|_{n,m}<s\}\) implies
\begin{equation}\label{div thm}
 nV(s)
 =
 s\int_{\{\|F\|_{n,m}=s\}}
 \bigl|\nabla_\Sigma \|F\|_{n,m}\bigr|\,\dd A_s .
\end{equation}
Indeed, on the level hypersurface \(\{\|F\|_{n,m}=s\}\), the outward unit normal is $\frac{\nabla_\Sigma\|F\|_{n,m}}{|\nabla_\Sigma\|F\|_{n,m}|},$ and
\[
        \nabla_\Sigma\frac{\|F\|_{n,m}^2}{2}
        =
        \|F\|_{n,m}\nabla_\Sigma\|F\|_{n,m}
        =
        s\nabla_\Sigma\|F\|_{n,m}.
\]
On the other hand, the coarea formula gives, for a.e. \(s\in(0,R_0)\),
\begin{equation}\label{coarea}
 V'(s)
 =
 \int_{\{\|F\|_{n,m}=s\}}
 \frac{1}{\bigl|\nabla_\Sigma \|F\|_{n,m}\bigr|}\,\dd A_s .
\end{equation}
Combining \eqref{div thm}, \eqref{coarea}, and
\eqref{decp}, we obtain
\[
        sV'(s)\le nV(s)
        \qquad\text{for a.e. }s\in(0,R_0).
\]
Equivalently, \(s^{-n}V(s)\) is non-increasing.

Since $\Sigma$ is smooth at the origin, the rescalings $s^{-1}\Sigma$
converge locally smoothly to the spacelike tangent plane $T_0\Sigma$.
On $T_0\Sigma$, the pseudo-ball $\big\{F\in T_0\Sigma:\,\|F\|_{n,m}<s\big\}$ is the Euclidean ball of radius $s$ with respect to the induced metric.
It follows that $\lim_{s\rightarrow 0^+}V(s)/s^n=|B^n|.$
By the monotonicity of $s^{-n}V(s)$, we have  
$V(s)\le |B^n| s^n .$ Letting $s\rightarrow R_0$, we get \eqref{claim}. Combining \eqref{max vol} and \eqref{claim}, we obtain
\[
 \int_A\sqrt{\det\big(I-(\D g_x)^\top \D g_x\big)}\dd x\le|B^n|R_0^n.
\]
Finally, letting $R_0\rightarrow  R$,  we establish \eqref{upbd s1lip}.
\end{proof}

\begin{remark}\label{plane equality}
If $L:\R^n\to\R^m$ is linear with $\|L\| <1$ and $A=\{x:|x|^2-|Lx|^2\le R^2\},$ then equality holds in \eqref{upbd s1lip}. This is the invariant equality example given by a spacelike plane through the origin. The codimension-one analogue of the above volume monotonicity formula can be found in Bartnik--Simon
\cite[(2.15)]{BS82}.
\end{remark}

Lemma \ref{Smooth graph estimate}  is available for global smooth maps. To apply it to general Lipschitz map defined only on a symmetric set, we
first recall an extension theorem.

\begin{lemma}\label{extension lemma}
Let $A=-A\subset\R^n$ and let $f:A\to\R^m$ be odd and $q$-Lipschitz. Then $f$ has a global odd $q$-Lipschitz extension.
\end{lemma}

\begin{proof}
By Kirszbraun's theorem \cite[Theorem 2.10.43]{Federer1969}, $f$ has a global $q$-Lipschitz extension $G_0$. Set
\[
 \tilde{f}(x)=\frac{G_0(x)-G_0(-x)}2.
\]
Then $\tilde f$ is odd and $q$-Lipschitz. If $x\in A$, then $-x\in A$ and $\tilde{f}(x)=f(x)$.
\end{proof}
With this extension, the strict Lipschitz case follows from the approximate argument.

\begin{proposition}\label{strict Lip case}
Let \(R>0\). Let $A=-A\subset\R^n$ be bounded and measurable, and let $f:A\to\R^m$ be odd and $q$-Lipschitz for some $q<1$. If
\[
 |x|^2-|f(x)|^2\le R^2,\qquad \forall\, x\in A,
\]
then \eqref{upperbd of 1-lip} holds.
\end{proposition}

\begin{proof}
Let $\tilde{f}$ be the extension from Lemma \ref{extension lemma}, and let $\eta_\varepsilon$ be a standard even mollifier. Let $\tilde{f}_\varepsilon=\tilde{f}*\eta_\varepsilon$. Then $\tilde{f}_\varepsilon$ is smooth, odd, and $q$-Lipschitz. Moreover, $\tilde{f}_\varepsilon\to \tilde{f}$ uniformly and $\D \tilde{f}_\varepsilon\to \D \tilde{f}$ almost everywhere. Define
\[
 R_\varepsilon^2=\sup_{x\in A}\bigl(|x|^2-|\tilde{f}_\varepsilon(x)|^2\bigr).
\]
By the uniform convergence we have $\lim_{\varepsilon\rightarrow 0}R_\varepsilon\le R.$ Now Proposition \ref{Smooth graph estimate}  yields
\[
 \int_A\sqrt{\det\big(I-(\D (\tilde{f}_\varepsilon)_x)^\top \D (\tilde{f}_\varepsilon)_x\big)}\dd x
 \le|B^n|R_\varepsilon^n.
\]
The integrands are bounded by $1$ and converge almost everywhere. By the dominated convergence theorem and the locality of approximate differentials, we obtain the desired estimate \eqref{upperbd of 1-lip}.
\end{proof}

\begin{proof}[Proof of Lemma \ref{upper bound of 1-lip}]
For $0<\varepsilon<1$, set $f_\varepsilon=(1-\varepsilon)f$. Since \eqref{pseudoball} implies $|f(x)|\le|x|,$ $f_\varepsilon$ is strictly Lipschitz. Define
\[
 R_\varepsilon^2
 :=\sup_{x\in A}\bigl(|x|^2-(1-\varepsilon)^2|f(x)|^2\bigr).
\]
By the uniform convergence we have $\lim_{\varepsilon\rightarrow 0}R_\varepsilon^2=\sup_{x\in A}(|x|^2-|f(x)|^2)\le R^2.$
Moreover,
\[
 0\leq I-(\D f_x)^\top \D f_x 
 \leq I-(1-\varepsilon)^2(\D f_x)^\top \D f_x.
\]
Hence, by Proposition \ref{strict Lip case}, we have
\[
 \int_A\sqrt{\det(I-(\D f_x)^\top \D f_x)}\dd x
 \le|B^n|R_\varepsilon^n.
\]
 The desired estimate \eqref{upperbd of 1-lip} follows by letting $\varepsilon\rightarrow 0$.
\end{proof}

\section{Proof of the geometric inequality}\label{main proof of ineq}

\begin{proof}[Proof of Theorem \ref{mainthm}]
If one of the sets has zero volume, we are done. By Lemma \ref{convexification}, it is enough to consider origin-symmetric convex bodies.

Since $n=1$ is trivial, it suffices to assume $n\ge2$. After the pseudo-isometry \eqref{new coordinate}, combining the \eqref{BXX bound}, \eqref{B monotonicity} with Proposition \ref{P projection}, we have that the optimal support $\Gamma$ is an odd, $1$-Lipschitz graph satisfying
\begin{equation}\label{upper bound}
        0\le |x|^2-|f(x)|^2\le 2C_k,
        \qquad \forall\,(x,f(x))\in\Phi(\Gamma).
\end{equation}
Applying Lemma \ref{upper bound of 1-lip} with $R=(2C_k)^{1/2}$, we obtain
\begin{equation}\label{up}
 \Vol_{\mathcal B}(\Gamma)\le|B^n|(2C_k)^{n/2}.
\end{equation}
On the other hand, Proposition \ref{lower vol} gives
\[
 \big(k(k-1)\big)^{n/2}
 \left(\prod_i|K_i|\right)^{1/k}
 \le\Vol_{\mathcal B}(\Gamma).
\]
Combining this with \eqref{up}, we have
\[
 \left(\prod_i|K_i|\right)^{1/k}
 \le|B^n|,
\]
which implies \eqref{main bound}.

If $K_1=\cdots=K_k=B^n$, then
\[
 \S_2(x_1,\ldots,x_k)
 \le\sum_{i<j}|x_i||x_j|
 \le \frac{k(k-1)}{2}=C_k,
\]
and the volume product is $|B^n|^k$, which equals the right-hand side of \eqref{main bound}. Hence, we finish the inequality case.
\end{proof}

\section{Geometric equality cases}\label{sec:geomeq}

Having proved the sharp inequality, we now turn to its equality cases in Theorem~\ref{mainthm}. The preceding argument shows that equality can occur only when every step in the volume comparison chain is sharp.

It suffices to classify equality for convex bodies. Indeed, suppose that the original measurable sets $K_i$ attain equality. Let $\widetilde K_i=\cl \operatorname{conv}(K_i)$. By Lemma \ref{convexification} and the inequality part of Theorem \ref{mainthm}, we obtain
\[
 \prod_i|K_i|\le\prod_i|\widetilde K_i|
 \le |B^n|^k
 =\prod_i|K_i|,
\] 
which implies that each $K_i$ agrees almost everywhere with the origin-symmetric convex body $\widetilde K_i$. Hence, we only need to consider convex bodies.

The case $n=1$ is trivial. The case \(k=2\) corresponds to the classical symmetric Blaschke--Santaló inequality and it's well known that equality is attained by a pair of polar dual ellipsoids. We now assume $n\ge2$ and $k\ge3$. 

Combining Section \ref{sec of bc and MA} with Section \ref{pse vol estimate}, we have 
\begin{equation}\label{final bd}
 (2C_k)^{n/2}\left(\prod_{i=1}^k |K_i|\right)^{1/k}
 \le
 \int_E \sqrt{\det G(p)}\dd p
 \le
 \Vol_{\mathcal B}(\Gamma)
 \le
 |B^n|(2C_k)^{n/2}.
\end{equation}
Here $2C_k=k(k-1)$. If equality  holds in \eqref{main bound}, then equality in the convexification step and in every inequality in \eqref{final bd} hold.  The only equality condition from this chain that we need is the rigidity in the pointwise determinant inequality of Lemma~\ref{det lemma}. Let $\nu$ be the Wasserstein barycenter of the probability measures
\[
 \mu_i=|K_i|^{-1}\1_{K_i}\dd x,
\]
and let $T_i=\nabla\varphi_i$ be the Brenier map from $\nu$ to $\mu_i$, as in Section~\ref{sec of bc and MA}. Recall that on $E=\{h>0\}$, we have
\[
 \det M_i(p)=|K_i|h(p)>0
 \qquad\text{for a.e. }p\in E .
\]
Thus the equality statement in Lemma~\ref{det lemma} implies
\[
 M_1(p)=\cdots=M_k(p)=\frac1k I
 \qquad\text{for a.e. }p\in E .
\]
Since $M_i(p)=k^{-1}\D T_i(p/k)$, we have
\begin{equation}\label{id matrix}
 \D T_i(z)=I
 \qquad \text{for}\,\,\nu\text{-a.e. }z .
\end{equation}
Then
\[
 1=\det\D T_i(z)=|K_i|\bar h(z)
 \qquad\text{for a.e. }z\in\{\bar h>0\}.
\]
Hence $|K_i|=|K_1|$, and
\begin{equation}\label{density of nv}
 \dd\nu=\alpha\1_{E}\dd z,
 \qquad
 \alpha=|K_1|^{-1},
\end{equation}
where  $E$ may be chosen origin-symmetric. So it suffices to prove that \eqref{id matrix}, \eqref{density of nv} and symmetry imply that the sets  $K_i$  are   equal.

\begin{lemma}\label{rigidity lem}
Let $K_1,\ldots,K_k\subset\R^n$ be origin-symmetric convex bodies, let $\mu_i=|K_i|^{-1}\1_{K_i}\dd x$, and let $\nu$ be an absolutely continuous, origin-symmetric probability measure.  Suppose that the Brenier maps $T_i=\nabla\varphi_i$ from $\nu$ to $\mu_i$ satisfy
\begin{equation}\label{bcenter}
 \D T_i=I\quad\nu\text{-}\mathrm{a.e.},
 \qquad \frac1k\sum_{i=1}^k T_i(z)=z\quad\nu\text{-}\mathrm{a.e.},
\end{equation}
and that $\dd\nu=\alpha\1_E\dd z$ for some origin symmetric measurable set $E$ and some $\alpha>0$. Then 
\[
 K_1=\cdots=K_k .
\]
\end{lemma}

\begin{proof}
Let \(T_i^*=\nabla\psi_i\) be the inverse Brenier map from \(\mu_i\) to \(\nu\).
Since \(T_i\) and \(T_i^*\) are a.e. inverse maps, and since convex potentials
are Alexandrov twice differentiable almost everywhere, we have, for
\(\nu\)-a.e. \(z\),
\[
 \D T_i^*\big(T_i(z)\big)\,\D T_i(z)=I .
\]
Because  \(\dd\nu=\alpha\mathbf 1_E\,\dd z\), the above equality holds for Lebesgue-a.e. \(z\in E\).  Since \((T_i)_\#\nu=\mu_i\),
this implies $\D T_i^*=I $ $\mu_i$-a.e. on $K_i.$ Thus the absolutely continuous part of \(\D^2\psi_i\) is \(I\,\dd x\) on \(\operatorname{int}K_i\). Since \(\D^2\psi_i\) is a positive semidefinite matrix-valued measure, \(\psi_i-\frac12|x|^2\) is convex on \(K_i\). 
Hence  the function \(\Psi_i^*(\cdot):=\sup_{x\in K_i}\ip{\cdot}{x}-\psi_i(x)\) is
differentiable and \(\nabla\Psi_i^*\) is \(1\)-Lipschitz. Moreover, 
\(\nabla\Psi_i^*=T_i\) \(\nu\)-a.e. After modifying \(T_i\) on a $\nu$-null set, we
may therefore assume that each \(T_i\) is \(1\)-Lipschitz.

Let \(E_0\subset E\) be a full \(\nu\)-measure set on which the barycenter
identity holds. For \(z,w\in E_0\), set
\[
 a_i:=T_i(z)-T_i(w),\qquad v:=z-w .
\]
Then $ \frac1k\sum_{i=1}^k a_i=v,$
while the \(1\)-Lipschitz property gives $|a_i|\le |v|$ for any $i$.
Hence
\[
 |v|
 =
 \left|\frac1k\sum_{i=1}^k a_i\right|
 \le
 \frac1k\sum_{i=1}^k |a_i|
 \le |v|.
\]
Thus equality holds throughout, and it follows that
$ a_i=v$ for every $i$.
Therefore, for each \(i\), there exists a vector \(c_i\) such that
\[
 T_i(z)=z+c_i\qquad \text{for}\,\,\nu\text{-a.e. }z .
\]

Combining this with \((T_i)_\#\nu=\mu_i\) and \(\dd\nu=\alpha\mathbf 1_E\,\dd z\), we have
\[
 |K_i|^{-1}\mathbf 1_{K_i}(x)\,\dd x
 =
 \alpha\,\mathbf 1_{E+c_i}(x)\,\dd x .
\]
Thus \(K_i=E+c_i\) up to a $\nu$-measure 0 set. 
Since \(K_i\) and $E$ are origin symmetric, we have \(c_i=0\) for all \(i\). Therefore $K_1=\cdots=K_k$.
\end{proof}

Applying Lemma \ref{rigidity lem} to \eqref{id matrix} and \eqref{density of nv}, we have $K_1=\cdots=K_k=:K.$ Taking $x_1=\cdots=x_k=x\in K$ in \eqref{constraint}, we obtain   $|x|^2\le 1$, and therefore $ K\subset \,B^n$. It follows from  this  and \eqref{final bd} that $K=  B^n.$ This completes the proof of   equality characterization in Theorem \ref{mainthm}.

\section{Equality in the functional inequality}\label{main proof of equality}

In this section, we mainly discuss the equality case in Theorem~\ref{thm:funct}. The proof follows the level-set reduction of Kalantzopoulos--Saroglou. The pointwise polarity condition forms a geometric constraint for the superlevel sets and we can apply Theorem~\ref{mainthm} to obtain the sharp volume bound for these sets. For our purposes, we use the following multiplicative form of Pr\'ekopa--Leindler. The classical Pr{\'e}kopa--Leindler inequality and its equality characterization have been widely studied; see e.g. \cite{Dubuc1977,Prekopa1971,Leindler1972,KW,BKX}. 

\begin{lemma}\label{mul-PL ineq}
Let $\Phi,\Phi_1,\ldots,\Phi_k:\R_+\to \R_+$ be measurable and integrable, and assume
\begin{equation}\label{ass of PL}
 \left(\prod_{i=1}^k\Phi_i(r_i)\right)^{1/k}
 \le \Phi\left(\left(\prod_{i=1}^kr_i\right)^{1/k}\right),
 \qquad \forall\,r_i>0.
\end{equation}
Then
\begin{equation}\label{PL}
\prod_{i=1}^k\int_0^\infty\Phi_i(r)\dd r
 \le\left(\int_0^\infty\Phi(r)\dd r\right)^k.
\end{equation}
Equality holds if and only if $t \mapsto e^t\Phi(e^t)$ is equivalent to a log-concave function, and there exist constants $a_i,b_i>0$ such that
\begin{equation}\label{aibi}
\prod_{i=1}^k a_i=\prod_{i=1}^k b_i=1
\end{equation}
and
\begin{equation}\label{Phi equality}
\Phi_i(r)=a_i\Phi(b_ir),
\qquad\mathrm{for\,\, a.e. }\,r>0.
\end{equation}
\end{lemma}

\begin{proof} 
 Denote
\begin{equation} \label{def  F}
 F_i(t)=e^t\Phi_i(e^t),\qquad F(t)=e^t\Phi(e^t),\qquad t\in\mathbb R.   
\end{equation} 
Then \eqref{ass of PL} becomes
\[
 \left(\prod_{i=1}^k F_i(t_i)\right)^{1/k}
 \le F\left(\frac1k\sum_{i=1}^k t_i\right).
\]
Hence \eqref{PL} follows from the classical Pr{\'e}kopa--Leindler inequality.

If equality in \eqref{PL} holds, by \cite[Theorem~4.1]{KW} or \cite[Theorem~28]{BKX}, there exist numbers $t_1,\ldots,t_k\in\mathbb R$ and constants $c_1,\ldots,c_k>0$ such that
\[
 F_i(t)=c_iF(t-t_i)
 \qquad\text{for a.e. }t .
\]
Recalling \eqref{def  F},  this is equivalent to saying that  
\[
\Phi_i(r)=a_i\Phi(b_ir)\qquad\text{for a.e. }r>0,
\]
where 
\[a_i = c_i\, e^{-t_i}, \qquad b_i = e^{-t_i}. \]
The condition  $ \sum_{i=1}^k t_i = 0 $ implies $  b_1 \cdots b_k=  1$.
Since equality holds in \eqref{PL}, through a direct computation, we have 
\[  \prod_{i=1}^k a_i=\prod_{i=1}^k b_i=  1.
\]
The converse follows from the log-concavity of $F$ and the same change of variables.
\end{proof}

Now, we are at a place to prove the equality case in Theorem~\ref{thm:funct}. When   $k=2$, we have the following characterization, which follows  from \cite[Proposition~3]{FM}.   
The only difference is that \cite[Proposition 3]{FM} is stated with 
\(\rho\) on \([0,\infty)\) and only for \(\langle x,y\rangle>0\); in our setting we apply it 
to \(\sqrt{\rho|_{[0,\infty)}}\), and 
notice that  \(f_1, f_2\) are even and  \(\rho\) is non-increasing. We therefore  omit its proof.

\begin{lemma}
Let \(\rho:\mathbb{R}\to\mathbb{R}^{+}\) be non-increasing, and let \(f_1,f_2:\mathbb{R}^{n}\to\mathbb{R}^{+}\) be even integrable functions satisfying
\[
    f_1(x)f_2(y)\leq \rho(\langle x,y\rangle),
    \qquad \forall\,x,y\in\mathbb{R}^{n}.
\]
Then 
\begin{equation}\label{equality at 2}
    \left(\int_{\mathbb{R}^{n}} f_1(x)\,\dd x\right)
    \left(\int_{\mathbb{R}^{n}} f_2(y)\,\dd y\right)
    =
    \left(
        \int_{\mathbb{R}^{n}}\rho(|u|^{2})^{1/2}\,\dd u
    \right)^{2} 
\end{equation}
holds if and only if there exist \(d>0\) and a symmetric
positive-definite $n\times n$ matrix \(Q\) such that
\begin{equation}\label{form}
        f_1(x)
    =
    d\,\rho(x^{\top }Qx)^{1/2},
    \qquad
    f_2(x)
    =
    d^{-1}\rho(x^{\top}Q^{-1}x)^{1/2}, \qquad\mathrm{for\,\, a.e. }\,x\in\mathbb R^n.
\end{equation}
\end{lemma}

Next, we deal with the equality case for Theorem \ref{thm:funct} with $k\geq 3$. 

\begin{theorem}
Let $\rho:\R\to\R_+$ be a non-increasing function and $f_1,\ldots,f_k:\R^n\to \R_+$ be even integrable functions such that
\begin{equation}\label{rho condition}
 \prod_{i=1}^kf_i(x_i)
 \le \rho\bigl(\S_2(x_1,\ldots,x_k)\bigr),
 \qquad \forall \,x_1,\cdots x_k\in\R^n.
\end{equation}
Then 
\begin{equation}\label{function equality}
 \prod_{i=1}^k\int_{\R^n}f_i \dd x= \left(\int_{\R^n}\rho\big(C_k|x|^2\big)^{1/k}\dd x\right)^k,  
\end{equation}
if and only if  there are constants $c_i>0$
such that
\begin{equation}\label{form of fi}
 \prod_{i=1}^kc_i=1,
 \qquad f_i(x)=c_i\rho\big(C_k|x|^2\big)^{1/k}\quad\mathrm{for \,\,a.e.}\,x\in\R^n,
\end{equation}
where
\begin{equation}\label{def of ci}
 c_i=\frac{\int_{\R^n}f_i \dd x}{\int_{\R^n}\rho\big(C_k|x|^2\big)^{1/k}\dd x}.
\end{equation}
\end{theorem}

\begin{proof}
Assume that equality in \eqref{function equality} holds. Denote $q_\rho(x)=\rho\big(C_k|x|^2\big)^{1/k}.$ Since \(f_i\in L^1(\R^n)\), combining this with \eqref{function equality}, we have that $q_\rho$ is integrable.
For all \(r,s>0\), set
\[
K_i(r)=\big\{x:f_i(x)>r\big\},\qquad \Phi_i(r)=|K_i(r)|,
\]
\[
Q(s)=\big\{x:q_\rho(x)>s\big\},\qquad \Phi(s)=|Q(s)|,
\]
and 
\[
 \tau_{\rho}(s)=\sup\big\{t\ge0:\rho(t)>s^k\big\},
\]
with the understanding $\sup\varnothing=0$. By the definition of $\tau_\rho$, we have
\begin{equation*}
 \Phi(s)=|B^n|\left(\frac{\tau_\rho(s)}{C_k}\right)^{n/2},
 \qquad s>0 .
\end{equation*}
In particular, $\tau_\rho(s)<\infty$ for every $s>0$, because $s\Phi(s)\le\int_{\R^n} q_\rho<\infty$.

We first verify \eqref{ass of PL}.  Fix $r_1,\ldots,r_k>0$ and put $s=(\prod_i r_i)^{1/k}$.  If $x_i\in K_i(r_i)$, then
\[
 s^k<\prod_i f_i(x_i)
 \le
 \rho\bigl(\S_2(x_1,\ldots,x_k)\bigr).
\]
Thus
\[
 \S_2(x_1,\ldots,x_k)\le \tau_\rho(s),
 \qquad x_i\in K_i(r_i).
\]
Here, if $\S_2(x_1,\ldots,x_k)<0$, the conclusion follows from $\tau_\rho(s)\geq0$.

If $\tau_{\rho}(s)>0$, let
\[
  K'_i
 :=\sqrt{\frac{C_k}{\tau_{\rho}(s)}}\,K_i(r_i),\qquad \forall \,1\le i\le k.
\]
By the quadratic homogeneity of \(\S_2\), 
\begin{equation}\label{recover con}
 \S_2(y_1,\ldots,y_k)\le C_k,
 \qquad \forall\,y_i\in K'_i.
\end{equation}
Then Theorem~\ref{mainthm} gives
\[
 \left(\prod_i\Phi_i(r_i)\right)^{1/k}
 =  \left(\prod_i|K_i(r_i)|\right)^{1/k} \le |B^n|
 \left(\frac{\tau_{\rho}(s)}{C_k}\right)^{n/2}
 =\Phi(s).
\]
If $\tau_{\rho}(s)=0$, then the same level sets satisfy
$\S_2\le\varepsilon$ for every $\varepsilon>0$.  Applying
Theorem~\ref{mainthm} to $ \sqrt{C_k/\varepsilon}\,K_i(r_i)$
 and letting $\varepsilon\to 0$ gives $ \left(\prod_i\Phi_i(r_i)\right)^{1/k}=0
 = \Phi(s).$

By the layer-cake representation,
\[
 \int_0^\infty\Phi_i(r)\dd r=\int_{\R^n}f_i(x)\,\dd x,
 \qquad
 \int_0^\infty\Phi(s)\dd s=\int_{\R^n}q_{\rho}(x)\dd x.
\]
By the equality case in Lemma
\ref{mul-PL ineq}, there exist $a_i,b_i>0$ such that
\begin{equation}\label{level reparam}
 \Phi_i(r)=a_i\Phi(b_ir)\quad\text{for \(\mathcal L^1\)-a.e. }r>0,,
 \qquad \prod^k_i a_i=\prod^k_i b_i=1.
\end{equation}
After the change of variables \(s=b_ir\), for each \(i\),
\begin{equation}\label{scal rigidity}
	\Phi_i(b_i^{-1}s)=a_i\Phi(s)
	\quad\text{for \(\mathcal L^1\)-a.e. }s>0.
\end{equation}

Fix such an \(s\) with
\(\Phi(s)>0\). Then \(\tau_\rho(s)>0\), and we set
\[
K_i''(s):=
\sqrt{\frac{C_k}{\tau_\rho(s)}}\,K_i(b_i^{-1}s).
\]
Since \(\prod_i b_i=1\), applying the same argument with \eqref{recover con}, we have
\[
\S_2(y_1,\ldots,y_k)\le C_k,
\qquad \forall\,y_i\in K_i''(s).
\]
Moreover, by \eqref{scal rigidity} and \eqref{Phi equality},
\[
 \begin{aligned}
 \prod_{i=1}^k|K''_i(s)|=
 \left(\frac{C_k}{\tau_{\rho}(s)}\right)^{kn/2}
 \prod_{i=1}^k|K_i(b_i^{-1}s)|=
 \left(\frac{C_k}{\tau_{\rho}(s)}\right)^{kn/2}
 \left(\prod^k_i a_i\right)\Phi(s)^k
 =|B^n|^k.
 \end{aligned}
\]
Thus equality holds in Theorem~\ref{mainthm}.  Since $k\ge3$, we have, up to an \(\mathcal L^n\)-null set,
\[
  K''_i(s)=B^n,
 \qquad i=1,\ldots,k.
\]
Consequently, for every \(i\) and for \(\mathcal L^1\)-a.e. \(s>0\) with \(\Phi(s)>0\),
\begin{equation}\label{ball of lev-set}
 K_i(b_i^{-1}s)
 =
 \sqrt{\frac{\tau_{\rho}(s)}{C_k}}\,B^n
 =Q(s),
 \quad\text{up to an \(\mathcal L^n\)-null set}.
\end{equation}
Comparing this with \eqref{scal rigidity}, we have  $a_i=1$ for all $i$. 

Fix   an $s$ with   $\Phi(s)=0$, such that \eqref{scal rigidity} holds. Then, \(|Q(s)|=\Phi(s)=0\) and
\[
|K_i(b_i^{-1}s)|
=
\Phi_i(b_i^{-1}s)
=
a_i\Phi(s)
=
0.
\]
Hence both sets are
\(\mathcal L^n\)-null. Combining this with  \eqref{ball of lev-set} gives
\begin{equation*}
 K_i(b_i^{-1}s)=Q(s), \quad\text{up to an \(\mathcal L^n\)-null set},
 \quad\text{for \(\mathcal L^1\)-a.e. }s>0 .
\end{equation*}
After the change of variable  $s=b_ir$, this becomes
\[
 \{f_i>r\}=\{q_{\rho}>b_ir\}
 =\{b_i^{-1}q_{\rho}>r\}
 \quad\text{for \(\mathcal L^1\)-a.e. }r>0.
\]
Then, the layer-cake representation for symmetric differences implies
\[
 \int_{\R^n}|f_i-b_i^{-1}q_{\rho}|\dd x
 =\int_0^\infty
   \bigl|\{f_i>r\}\mathbin{\triangle}
          \{b_i^{-1}q_{\rho}>r\}\bigr|\dd r
 =0.
\]
Hence $f_i=b_i^{-1}q_{\rho}$ almost everywhere.  Taking integrals and 
$c_i=b_i^{-1}$, we obtain \eqref{form of fi} and  \eqref{def of ci}.

Conversely, assume that the tuple already satisfies
\eqref{rho condition} and \eqref{form of fi}. Then
\[
 \prod_{i=1}^k\int_{\R^n}f_i(x)\, \dd x
 =\prod_{i=1}^k\left(c_i\int_{\R^n}q_{\rho}(x)\dd x\right)
 =\left(\int_{\R^n}q_{\rho}(x)\dd x\right)^k.
\]
Therefore equality holds in \eqref{function equality}.
\end{proof}

\end{document}